\documentclass[reqno,11pt]{amsart}

\usepackage{amsmath,amsfonts,amsthm,amssymb,amscd}
\usepackage[dvips]{graphics}
\usepackage{epsfig}
\usepackage[all,cmtip]{xy}

\DeclareMathOperator{\Hom}{Hom}

\DeclareMathOperator{\id}{id}

\DeclareMathOperator{\Spec}{Spec}

\DeclareMathOperator{\Tot}{Tot}

\DeclareMathOperator{\KH}{KH}

\DeclareMathOperator{\K}{K}

\DeclareMathOperator{\HH}{HH}

\DeclareMathOperator{\HC}{HC}

\DeclareMathOperator{\BOT}{BOT}

\DeclareMathOperator{\holim}{holim}

\newcommand{\Spectra}{\mathsf{Spectra}}

\newcommand\be{\begin{equation}}
\newcommand\ee{\end{equation}}
\newcommand\bea{\begin{eqnarray}}
\newcommand\eea{\end{eqnarray}}
\newcommand\bi{\begin{itemize}}
\newcommand\ei{\end{itemize}}
\newcommand\ben{\begin{enumerate}}
\newcommand\een{\end{enumerate}}
\newcommand\bc{\begin{center}}
\newcommand\ec{\end{center}}
\newcommand\ba{\begin{array}}
\newcommand\ea{\end{array}}

\newcommand{\R}{\ensuremath{\mathbb{R}}}

\newcommand{\Z}{\ensuremath{\mathbb{Z}}}
\newcommand{\Q}{\mathbb{Q}}
\newcommand{\N}{\mathbb{N}}

\newcommand{\G}{\ensuremath{\mathbb{G}}}
\newcommand{\Prj}{\ensuremath{\mathbb{P}}}
\newcommand{\Hyp}{\ensuremath{\mathbb{H}}}

\newtheorem{thm}{Theorem}[section]

\newtheorem{const}[thm]{Construction}
\newtheorem{cor}[thm]{Corollary}
\newtheorem{lem}[thm]{Lemma}
\newtheorem{prop}[thm]{Proposition}
\newtheorem{exa}[thm]{Example}
\newtheorem{defi}[thm]{Definition}

\theoremstyle{definition}
\newtheorem{rek}[thm]{Remark}

\numberwithin{equation}{section}

\newcommand{\inv}{^{-1}}

\begin{document}

\title[$\KH$-Theory of Complete Simplicial Toric Varieties]{The $\KH$-Theory of Complete Simplicial Toric Varieties and the Algebraic $\K$-Theory of Weighted Projective Spaces}

\author{Adam Massey}
\email{amassey3102@ucla.edu}
\address{Mathematics Department, University of California, Los Angeles, Los Angeles, CA 90095}

\date{\today}


\begin{abstract}
We show that, for a complete simplicial toric variety $X$, we can determine its homotopy $\KH$-theory entirely in terms of the torus pieces of open sets forming an open cover of $X$.  We then construct conditions under which, given two complete simplicial toric varieties, the two spectra $\KH(X) \otimes \Q$ and $\KH(Y) \otimes \Q$ are weakly equivalent.  We apply this result to determine the rational $\KH$-theory of weighted projective spaces.  We next examine $\K$-regularity for complete toric surfaces; in particular, we show that complete toric surfaces are $\K_{0}$-regular.  We then determine conditions under which our approach for dimension $2$ works in arbitrary dimensions, before demonstrating that weighted projective spaces are not $\K_{1}$-regular, and for dimensions bigger than $2$ are also not in general $\K_{0}$-regular.
\end{abstract}

\maketitle


\section{Introduction}\label{sect:intro}

In this paper, we examine the basic properties of complete simplicial toric varieties, and use these properties to compute their $\K$-theory.  Following the ideas of \cite[Proposition $5.6$]{CHWW}, the $\K$-theory of any toric variety over a field of characteristic $0$ is obtained as a direct sum of its $\KH$-theory and its $\mathcal{F}_{\K}$ groups.  We begin by focusing on the $\KH$-theory; the $\mathcal{F}_{\K}$ groups turn out to be much more difficult and are handled separately.


We conjecture that if $X$ and $Y$ are two complete simplicial toric varieties with the same simplicial structure, then the $\KH$-theories of $X$ and $Y$ are isomorphic; however, we are currently unable to prove a relationship between $\KH(X)$ and $\KH(Y)$ exists under these conditions.  To remedy this, we add additional conditions to force a relationship between $X$ and $Y$ that allows us to show that $\KH(X)$ and $\KH(Y)$ are rationally isomorphic.  The main goal of the first half of this paper is to construct and examine this relationship.



Using this relationship, our goal throughout this paper will be to prove the following theorem:

\begin{thm}\label{bigmain1}

Let $\Prj(q_{0} , ... , q_{d})$ be any weighted projective space of dimension $d$ over a regular ring $R$.  Then: \begin{enumerate}

\item[(a)] For every $n$, we have $\KH_{n}(\Prj(q_{0} , ... , q_{d})) \otimes \Q \cong \KH_{n}(\Prj^{d}) \otimes \Q.$

\end{enumerate} \noindent If $R$ happens to be a field of characteristic $0$, then we can conclude the following additional results: \begin{enumerate}

\item[(b)] Any $2$-dimensional weighted projective space $\Prj(a,b,c)$ is $\K_{0}$-regular.

\item[(c)] Any $d$-dimensional weighted projective space $\Prj(q_{0} , ... , q_{d})$ whose singular set consists of only isolated singular points is $\K_{0}$-regular.

\item[(d)] If our weighted projective space is of the form $\Prj(1,1, ... ,1,a)$, then for $n \leq 0$ we can conclude the stronger statement $\K_{n}(\Prj(1,1, ... ,1,a)) \cong \K_{n}(\Prj^{d}).$

\end{enumerate}

\end{thm}


Theorem \ref{bigmain1} is proven in several stages.  Most of our work towards proving Theorem \ref{bigmain1} is done by proving Theorem \ref{main1}; as such, much of this paper will focus on the proof of (and applications of) Theorem \ref{main1}.  Our work is related to a recent paper of Gubeladze's, in which he proves that for a projective, simplicial toric variety $X$ over a regular ring $R$, $\KH_{n}(X)\otimes \Q = (\KH_{n}(R)\otimes \Q)^{m}$ where $m$ is the number of maximal cones in the fan of $X$; see \cite[Corollary 2.5(c)]{Gub2}.  This provides an alternate proof of part (a) of Theorem \ref{bigmain1} and, as we'll see later, provides an alternate proof of Theorem \ref{main1} in the case that our toric varieties are projective.  The interested reader is encouraged to read \cite{Gub2} for further details.

\section{Toric Varieties: Notations and Terminology}\label{sect:toric}

For this paper we adopt the notation in \cite{Ful} for toric varieties: $\sigma$ denotes a cone, $\tau \prec \sigma$ denotes a face of $\sigma$, the $R$-scheme $U_{\sigma} = \Spec \left( R[\sigma^{\vee} \cap M] \right)$ denotes the affine open set associated to the cone $\sigma$ (where $R$ is any ring), $\Delta$ denotes a fan (and when we want to specify the fan associated to the toric variety $X$ we use the notation $\Delta_{X}$), and $X(\Delta)$ denotes the toric variety associated to the fan $\Delta$.  For the details of these constructions, the reader is referred to the standard references \cite{Ful} and \cite{Cox}, as well as the recent preprint \cite{Roh}.  One can also find many of the basics that we assume for this paper discussed in the early sections of \cite{CHWW}.  Note that we only consider split toric varieties; the case in which the torus is not split will not be considered in this paper.

We say that a cone $\sigma$ is \emph{simplicial} if its minimal set of generators $\{v_{1}, ... , v_{k}\}$ is linearly independent over $\R$.  We say that a fan $\Delta$ is simplicial if every cone $\sigma \in \Delta$ is simplicial, and we say that a toric variety $X$ is simplicial if its associated fan $\Delta_{X}$ is simplicial.



Let $N$ be a lattice, and $\sigma \subset N_{\R}$ be a cone.  We define $N_{\sigma} = (\sigma \cap N) + (- \sigma \cap N)$ to be the sublattice of $N$ generated by $\sigma$.  Similarly, we define $\widetilde{N}_{\sigma} = N/N_{\sigma}$.





\begin{prop}\label{toruspart1}

Let $\sigma$ be a $p$-dimensional cone in a lattice $N$ with $\dim N_{\R} = n$.  Then $U_{\sigma} \cong U_{\sigma'} \times T$, where $T \cong \G_{m}^{n-p}$ is a split algebraic torus of rank $n-p$.

\end{prop}

\begin{proof}

See \cite[Page $29$]{Ful}.




\end{proof}

\begin{defi}\label{toruspart2}

We define the ``torus piece of $U_{\sigma}$'' to be $\Spec (R[\Hom(\widetilde{N}_{\sigma},\Z)])$, and we denote this torus piece by $T_{\sigma}$.  When we want to indicate that $T_{\sigma}$ is the torus piece of an open subset $U_{\sigma}$ of the variety $X$, we will do so by writing it as $T_{\sigma}^{X}$.

\end{defi}

\begin{rek}

It is an easy exercise to check that the torus $T$ constructed in Proposition \ref{toruspart1} is the torus piece $T_{\sigma}$ of $U_{\sigma}$ in the sense of Definition \ref{toruspart2}.


\end{rek}

\section{The Simplicial Structure of a Complete Simplicial Toric Variety}\label{sect:simp}

From \cite[Proposition $5.6$]{CHWW}, the algebraic $\K$-theory of any toric variety $X$ (over a field of characteristic $0$) is determined completely by its $\KH$-theory and its $\mathcal{F}_{\K}$ groups.  We begin by examining the $\KH$-theory of complete simplicial toric varieties.  






\begin{prop}\label{prop:aff}

Let $\sigma$ be any $p$-dimensional cone.  Then \bea \K(T_{\sigma}) \longrightarrow \KH(T_{\sigma}) \longrightarrow \KH(U_{\sigma}) \eea are weak equivalences as spectra.  In other words, the $\KH$ groups of an open set corresponding to a cone are just the $\K$ groups of its associated torus piece.

\end{prop}

\begin{proof}

This is immediate from \cite[Theorem $1.2$, Part (a)]{Wei1} and the fact that algebraic tori are smooth.


\end{proof}

It is the result of Proposition \ref{prop:aff} that provides the intuition for the approach we use, and leads us to consider ways in which we might use the simplicial structure of a toric variety to determine its $\KH$-theory.  The next few sections will be dedicated to proving and applying the following theorem, which is built from the intuition of Proposition \ref{prop:aff} and is our main technical result.

\begin{thm}\label{main1}

Let $X$ and $Y$ be two complete $d$-dimensional simplicial toric varieties over a regular ring $R$. Let $\Delta_{X}$ live in the lattice $N^{X}$ and $\Delta_{Y}$ live in the lattice $N^{Y}$ and suppose that $\Delta_{X}$ and $\Delta_{Y}$ are rationally isomorphic via a rational linear automorphism $F_{\R}: N^{X}_{\R} \longrightarrow N^{Y}_{\R}$.  Then $\KH(X) \otimes \Q$ and $\KH(Y) \otimes \Q$ are weakly equivalent as spectra; in particular, $\KH_{n}(X) \otimes \Q \cong \KH_{n}(Y) \otimes \Q$ for all $n$.

\end{thm}

There are several stages that go into the proof of Theorem \ref{main1}.  First, we extend the intuition of Proposition \ref{prop:aff} by proving that, for a complete simplicial toric variety $X$, $\KH(X)$ is determined by the $\K$-theory of algebraic tori (a simplicial scheme we call $\BOT_{X}$; see Definition \ref{BOTdef}). Then we determine a relationship between $\BOT_{X}$ and $\BOT_{Y}$ (for two complete simplicial toric varieties $X$ and $Y$), and use this to determine a relationship between $\KH(X) \otimes \Q$ and $\KH(Y) \otimes \Q$. We begin this construction in Section \ref{sect:BOT}.


\section{The Construction of the Simplicial Scheme $\BOT_{X}$}\label{sect:BOT}

As it turns out, simply knowing that $X$ and $Y$ have the same simplicial structure is not enough to determine that they have the same $\KH$-theory.  However, since $\KH$ satisfies Zariski descent, we can approach this problem from the perspective of simplicial schemes.  This is our next goal.


\begin{const}\label{simpscheme}

Let $X$ be a complete simplicial toric variety.  Then $X$ gives rise to a simplicial scheme, which we call $\mathcal{U}_{X}$.  To construct the simplicial scheme structure, we simply take the Zariski cover given by open sets associated to maximal cones in the fan of $X$, and then we take the \v{C}ech nerve of that cover.  For the remainder of this paper, when we say ``simplicial scheme associated to $X$'', we are referring to the simplicial scheme $\mathcal{U}_{X}$.


\end{const}

To extend the intuition of Proposition \ref{prop:aff}, we build a new simplicial scheme consisting of torus pieces associated to open sets of the form $U_{\sigma}$.  We begin with some new definitions.


\begin{defi}\label{facedegendef}

Let $X$ be a complete simplicial toric variety with associated simplicial scheme $\mathcal{U}_{X}$.  Let $\Delta_{X} \subset N^{X}$, and let $\widetilde{N}_{\tau}^{X} = N^{X}/N_{\tau}^{X}$.  We define two lattice maps, which we call $\widetilde{d_{j}}$ and $\widetilde{s_{j}}$, by the following construction.  Let $\tau = \sigma_{0} \cap \cdots \cap \sigma_{n}$ and $\tau_{j} = \sigma_{0} \cap \cdots \cap \widehat{\sigma_{j}} \cap \cdots \cap \sigma_{n}$.  Then $\widetilde{d_{j}}: \widetilde{N}_{\tau}^{X} \longrightarrow \widetilde{N}_{\tau_{j}}^{X}$ by lifting $\widetilde{N}_{\tau}^{X}$ to $N^{X}$, mapping $N^{X}$ to itself via the identity (with $\tau \longrightarrow \tau_{j}$ via inclusion), and taking canonical surjection onto $\widetilde{N}_{\tau_{j}}^{X}$.  Similarly, construct $\widetilde{s_{j}}: \widetilde{N}_{\tau}^{X} \longrightarrow \widetilde{N}_{\tau}^{X}$ by lifting $\widetilde{N}_{\tau}^{X}$ to $N^{X}$, mapping $N^{X}$ to itself via the identity (with $\tau \longrightarrow \tau$ via the identity), and taking canonical surjection onto $\widetilde{N}_{\tau}^{X}$.  Observe that, with this construction, $\widetilde{s_{j}}$ is the identity map.

\end{defi}

\begin{defi}\label{BOTdef}

Let $X$ be a complete simplicial toric variety with associated simplicial scheme $\mathcal{U}_{X}$ as shown in Construction \ref{simpscheme}.  We define a new simplicial scheme, which we call $\BOT_{X}$, by the following properties:

\begin{enumerate}

\item We define $(\BOT_{X})_{n} = \coprod T_{\alpha(\sigma_{0}, \cdots , \sigma_{n})}$, where $T_{\alpha(\sigma_{0}, \cdots , \sigma_{n})}$ is the associated torus piece for the open set $U_{\sigma_{0}} \cap \cdot \cdot \cdot \cap U_{\sigma_{n}}$ (in the sense of Definition \ref{toruspart2}) and $\alpha(\sigma_{0}, \cdots , \sigma_{n})$ is its rank.  As before, the $\sigma_{i}$'s are all maximal cones.

\item We define the face and degeneracy maps, denoted $d_{j}^{\BOT_{X}}$ and $s_{j}^{\BOT_{X}}$ respectively, to component-wise be the morphisms of toric varieties that are induced by the lattice maps $\widetilde{d_{j}}$ and $\widetilde{s_{j}}$ of Definition \ref{facedegendef}.
\end{enumerate}

\end{defi}

\begin{thm}\label{BOTproof}

$\BOT_{X}$, as defined in Definition \ref{BOTdef}, is a simplicial scheme.

\end{thm}

\begin{proof}

Definition \ref{BOTdef} has already given us our objects $(\BOT_{X})_{n}$ and our face and degeneracy maps, and it is an easy exercise to check that $d_{j}^{\BOT_{X}}$ and $s_{j}^{\BOT_{X}}$ are well-defined and satisfy the usual simplicial identities.

\end{proof}


Let $q_{n}^{\mathcal{U_{X}}}$ be the morphism of schemes obtained by projecting each component of $(\mathcal{U}_{X})_{n}$ onto its torus piece.  By it's very definition, this gives us a morphism $q_{n}^{\mathcal{U_{X}}} : (\mathcal{U}_{X})_{n} \longrightarrow (\BOT_{X})_{n}$ and we have:

\begin{thm}\label{morph}

The morphism $q^{\mathcal{U_{X}}} : \mathcal{U}_{X} \longrightarrow \BOT_{X}$ given in degree $n$ by $q_{n}^{\mathcal{U_{X}}} : (\mathcal{U}_{X})_{n} \longrightarrow (\BOT_{X})_{n}$ is a morphism of simplicial schemes.

\end{thm}

\begin{proof}

This is straightforward by the construction of the maps $d_{j}^{\BOT_{X}}$ and $s_{j}^{\BOT_{X}}$.

\end{proof}

\begin{thm}\label{Weq}

The morphisms $\KH(q^{\mathcal{U_{X}}}) : \KH(\BOT_{X}) \longrightarrow \KH(\mathcal{U}_{X})$ and $\KH(q^{\mathcal{U_{X}}}) \otimes \id_{\Q} : \KH(\BOT_{X}) \otimes \Q \longrightarrow \KH(\mathcal{U}_{X}) \otimes \Q$ of cosimplicial spectra induced by the map $q^{\mathcal{U_{X}}}$ of Theorem \ref{morph} are weak equivalences.

\end{thm}

\begin{proof}

This follows by applying Proposition \ref{prop:aff} componentwise.

\end{proof}

\begin{cor}\label{holim}

The morphisms $\holim(\KH(\BOT_{X})) \longrightarrow \holim(\KH(\mathcal{U}_{X}))$ and $\holim(\KH(\BOT_{X}) \otimes \Q) \longrightarrow \holim(\KH(\mathcal{U}_{X}) \otimes \Q)$ are weak equivalences of $\Spectra$.

\end{cor}

\begin{proof}

This is immediate from Theorem \ref{Weq} and the construction of the $\holim$ functor given in \cite{BK}.  For our purposes, $\holim(-) = \Tot(\Pi^{*}(R(-)))$, where $R$ denotes fibrant replacement, $\Pi^{*}$ denotes cosimplicial replacement, and $\Tot$ is the total object functor.

\end{proof}

This allows us to conclude the following important consequence for $\BOT_{X}$.

\begin{thm}\label{weaklyequiv}

Let $X$ be a complete simplicial toric variety over a ring $R$.  Then the two spectra $\holim(\KH(\BOT_{X}))$ and $\KH(X)$ are weakly equivalent.

\end{thm}

\begin{proof}

By Corollary \ref{holim} and the fact that $\KH$ satisfies Zariski descent, we get a zig-zag \bea \holim(\KH(\BOT_{X})) \longrightarrow \holim(\KH(\mathcal{U}_{X})) \longleftarrow \KH(X) \eea where each map is a weak equivalence; thus the spectra $\holim(\KH(\BOT_{X}))$ and $\KH(X)$ are weakly equivalent as claimed.

\end{proof}

Theorem \ref{weaklyequiv} completes our generalization of the intuition we presented at the end of Section \ref{sect:simp}, showing that $\KH_{n}(X)$ is indeed determined only by the torus pieces of open sets associated to maximal cones covering $X$ as claimed.



\section{The Proof of Theorem \ref{main1}}\label{sect:main}

With the work of Sections \ref{sect:simp} and \ref{sect:BOT}, we are now ready to prove Theorem \ref{main1}.  Section \ref{sect:BOT} focused on the construction of $\BOT_{X}$ for a given complete simplicial toric variety $X$.  Similarly, we could construct $\BOT_{Y}$ for a different complete simplicial toric variety $Y$.  The question is: when are these two simplicial schemes related?  In general this is not known; however, if we impose the conditions that the fans of $X$ and $Y$ are rationally isomorphic via a rational linear automorphism $F_{\R} : N_{\R}^{X} \longrightarrow N_{\R}^{Y}$, and $X$ and $Y$ are defined over a regular ring $R$, then we get a very useful relationship.


\begin{lem}\label{propres}

Let $X$ and $Y$ be two complete simplicial toric varieties whose fans are isomorphic via a rational linear automorphism $F_{\R} : N_{\R}^{X} \longrightarrow N_{\R}^{Y}$ and denote the induced mapping on cones by $\varphi$ (so that $\sigma \mapsto \varphi(\sigma)$).  Then the induced map $\left(\widetilde{F}_{\sigma} \right)_{\R} : \left(\widetilde{N}_{\sigma}^{X} \right)_{\R} \longrightarrow \left( \widetilde{N}_{\varphi (\sigma)}^{Y} \right)_{\R}$ is also a rational linear automorphism.



\end{lem}

\begin{proof}

One need only show that the induced map $\widetilde{F}_{\sigma} : \widetilde{N}_{\sigma}^{X} \longrightarrow \widetilde{N}_{\varphi (\sigma)}^{Y}$ is injective with finite cokernel, which is an easy consequence of the Snake Lemma.


\end{proof}


\begin{thm}\label{isogeny}

Let $X$ and $Y$ be two complete simplicial toric varieties whose fans are isomorphic via a rational linear automorphism $F_{\R} : N_{\R}^{X} \longrightarrow N_{\R}^{Y}$ and denote the induced mapping on cones by $\varphi$ (so that $\sigma \mapsto \varphi(\sigma)$).  Then $F$ (obtained from $F_{\R}$ by clearing denominators) induces a morphism of simplicial schemes $\BOT_{X} \longrightarrow \BOT_{Y}.$  Furthermore, for every $n$, the induced morphism $(\BOT_{X})_{n} \longrightarrow (\BOT_{Y})_{n}$ is, in each component, an isogeny.


\end{thm}

\begin{proof}

As we've already seen, $\widetilde{N}_{\sigma}^{X}$ and $\widetilde{N}_{\varphi (\sigma)}^{Y}$ determine the torus pieces of $U^{X}_{\sigma}$ and $U^{Y}_{\varphi (\sigma)}$, and so $\widetilde{F}_{\sigma} : \widetilde{N}_{\sigma}^{X} \longrightarrow \widetilde{N}_{\varphi (\sigma)}^{Y}$ determines a morphism between the tori $T^{X}_{\sigma}$ and $T^{Y}_{\varphi (\sigma)}$.  Since $\widetilde{F}_{\sigma}$ is injective with finite cokernel (Lemma \ref{propres}), the induced map on the corresponding tori is an isogeny; therefore, the morphism $(\BOT_{X})_{n} \longrightarrow (\BOT_{Y})_{n}$ is, in each component, an isogeny.  The fact that the maps $(\BOT_{X})_{n} \longrightarrow (\BOT_{Y})_{n}$ commute with the face and degeneracy maps is an easy exercise left to the reader.

\end{proof}




\begin{cor}\label{inducedmap1}

The morphism of Theorem \ref{isogeny} induces a morphism of cosimplicial spectra $\KH(\BOT_{Y}) \longrightarrow \KH(\BOT_{X})$ which is, component-wise in each degree, given by induced morphisms $f^{*}$ where $f$ is an isogeny of tori.  Similarly, the morphism of Theorem \ref{isogeny} induces a morphism of cosimplicial spectra $\KH(\BOT_{Y}) \otimes \Q \longrightarrow \KH(\BOT_{X}) \otimes \Q$ which is, component-wise in each degree, given by induced morphisms $(f^{*})_{\Q}$.

\end{cor}

Our next goal is to show that, if our base ring is in fact regular, the morphism of cosimplicial spectra $\KH(\BOT_{Y}) \otimes \Q \longrightarrow \KH(\BOT_{X}) \otimes \Q$ induced by the morphism of Theorem \ref{isogeny} is a weak equivalence. This follows from the following theorem (suggested by an anonymous referee):

\begin{thm}\label{fixedpoint4}

Let $R$ be a regular ring, $d$ a natural number, and $\alpha: \Z^{d} \longrightarrow \Z^{d}$ an injective group homomorphism.  Then the induced morphism \bea \alpha_{n}: \K_{n}\left( R[\Z^{d}] \right) \otimes \Q \longrightarrow \K_{n}\left( R[\Z^{d}] \right) \otimes \Q \nonumber \eea is an automorphism.  As a consequence, if $f$ is an isogeny of algebraic tori, then $(f^{*})_{\Q}$ is an isomorphism.

\end{thm}

\begin{proof}

Begin by noting that, via a standard diagonalization argument, we may assume without loss of generality that $\alpha (e_{i}) = m_{i} e_{i}$ for $i=1,...,d$, where $m_{i} \in \N$ for all $i$, and the $e_{i}$ are the standard basis vectors.  By inductively applying the Fundamental Theorem, we have \bea \K_{n}\left(R[\Z^{d}]\right) = \K_{n}\left( R[t_{1}, t_{1}\inv, ..., t_{d}, t_{d}\inv] \right) = \bigoplus_{r \leq d} \{t_{j_{1}}, t_{j_{2}},..., t_{j_{r}} \} \cdot \K_{n-r}(R) \nonumber \eea where $\{t_{a} , t_{b}\}$ denotes the cup product of $[t_{a}]$ and $[t_{b}]$ (viewed as elements of $\K_{1}\left( R[t_{1}, t_{1}\inv, ..., t_{d}, t_{d}\inv] \right)$ as usual).  Then on each graded component of $\K_{n}\left( R[\Z^{d}] \right)$, $\alpha$ induces the morphism \bea \{t_{j_{1}}, t_{j_{2}},..., t_{j_{r}} \} \cdot y & \mapsto & \{t_{j_{1}}^{m_{j_{1}}}, t_{j_{2}}^{m_{j_{2}}},..., t_{j_{r}}^{m_{j_{r}}} \} \cdot y \nonumber \\
& = & m_{j_{1}}m_{j_{2}} \cdots m_{j_{r}} \left( \{t_{j_{1}}, t_{j_{2}},..., t_{j_{r}} \} \cdot y \right) \nonumber \eea which implies that $\alpha$ induces a morphism which looks like multiplication by $m_{1}^{\epsilon_{1}} \cdots m_{d}^{\epsilon_{d}}$ on each graded component (where $\epsilon_{i} \in \{0, 1\}$ for each $i$).  Tensoring with $\Q$ yields that $\alpha_{n}$ is an isomorphism as claimed.

\end{proof}

With our above work, we are now ready to prove Theorem \ref{main1}.

\begin{proof}[Proof of Theorem \ref{main1}]

Suppose that $X$ and $Y$ are two complete, simplicial toric varieties over a regular ring $R$, and suppose $\Delta_{X}$ and $\Delta_{Y}$ are isomorphic via a rational linear automorphism $F_{\R} : N_{\R}^{X} \longrightarrow N_{\R}^{Y}$.  We know from Theorem \ref{isogeny} that $F$ induces a morphism $\BOT_{X} \longrightarrow \BOT_{Y}$ which is component-wise in each degree an isogeny.  By Theorem \ref{fixedpoint4}, the induced maps $(f^{*})_{\Q}$ are component-wise in each degree isomorphisms.


By Corollary \ref{inducedmap1}, the induced morphism of cosimplicial spectra $\KH(\BOT_{Y}) \otimes \Q \longrightarrow \KH(\BOT_{X}) \otimes \Q$ is in each degree component-wise given by morphisms of the form $(f^{*})_{\Q}$ and as a consequence, the morphism $\KH(\BOT_{Y}) \otimes \Q \longrightarrow \KH(\BOT_{X}) \otimes \Q$ is a weak equivalence of cosimplicial spectra; therefore the morphism $\holim(\KH(\BOT_{Y}) \otimes \Q) \longrightarrow \holim(\KH(\BOT_{X}) \otimes \Q)$ is a weak equivalence of spectra as well.  This gives us the following diagram: $$
\xymatrix{ \holim(\KH(\BOT_{Y}) \otimes \Q) \ar[r]^-{\sim} \ar[d]_-{\sim} & \holim(\KH(\BOT_{X}) \otimes \Q) \ar[d]^-{\sim} \\
\holim(\KH(\mathcal{U}_{Y}) \otimes \Q)  &  \holim(\KH(\mathcal{U}_{X}) \otimes \Q) \\
\KH(Y) \otimes \Q  \ar[u]^-{\sim} & \KH(X) \otimes \Q \ar[u]_-{\sim} }
$$ where the morphisms $\holim(\KH(\BOT_{Y}) \otimes \Q) \longrightarrow \holim(\KH(\mathcal{U}_{Y}) \otimes \Q)$ and $\holim(\KH(\BOT_{X}) \otimes \Q) \longrightarrow \holim(\KH(\mathcal{U}_{X}) \otimes \Q)$ are weak equivalences by Corollary \ref{holim}, and the morphisms $\KH(Y) \otimes \Q \longrightarrow \holim(\KH(\mathcal{U}_{Y}) \otimes \Q)$ and $\KH(X) \otimes \Q \longrightarrow \holim(\KH(\mathcal{U}_{X}) \otimes \Q)$ are weak equivalences by the fact that $\KH(-) \otimes \Q$ satisfies Zariski descent.  So the spectra $\KH(X) \otimes \Q$ and $\KH(Y) \otimes \Q$ are weakly equivalent, establishing $\KH_{n}(X) \otimes \Q \cong \KH_{n}(Y) \otimes \Q$ for all $n$.  This completes the proof.
%

\end{proof}

\section{Applications of Theorem \ref{main1}}\label{sect:applications}

With Theorem \ref{main1} proven, we can now apply it to calculate the groups $\KH_{n}(\Prj (q_{0}, ..., q_{d})) \otimes \Q$, by comparing them to $\KH_{n}(\Prj^{d}) \otimes \Q$.  

\begin{thm}\label{ratiso}

Let $X$ and $Y$ be two complete $d$-dimensional simplicial toric varieties over a regular ring $R$, such that the number of maximal cones in both $\Delta_{X}$ and $\Delta_{Y}$ is $d+1$.  Then $\KH_{n}(X) \otimes \Q \cong \KH_{n}(Y) \otimes \Q$ for all $n$.

\end{thm}

\begin{proof}

It is an easy exercise to see that if $X$ and $Y$ be two complete $d$-dimensional simplicial toric varieties such that the number of maximal cones in both $\Delta_{X}$ and $\Delta_{Y}$ is $d+1$, then $\Delta_{X}$ and $\Delta_{Y}$ are rationally isomorphic; let $F_{\R}$ be that isomorphism.  Then the result follows by Theorem \ref{main1}.

\end{proof}

\begin{rek}

One can actually prove an alternate form of Theorem \ref{ratiso} using \cite{Gub2}.  Indeed, it is an easy exercise to show that the $X$ and $Y$ given in Theorem \ref{ratiso} are projective.  Therefore, by \cite[Corollary 2.5(c)]{Gub2}, $\KH_{n}(X)\otimes \Q \cong (\K_{n}(R)\otimes \Q)^{d+1} \cong \KH_{n}(Y)\otimes \Q$.

\end{rek}

\begin{cor}\label{main2}

If $\Prj (q_{0}, ..., q_{d})$ is any $d$-dimensional weighted projective space defined over a regular ring $R$, then $\KH_{n}(\Prj (q_{0}, ..., q_{d})) \otimes \Q \cong \KH_{n}(\Prj^{d}) \otimes \Q$ for all $n$, establishing part (a) of Theorem \ref{bigmain1}.


\end{cor}






We can use Theorem \ref{main2} and cdh-descent to calculate the $\KH$-theory (up to degree $0$) of weighted projective spaces of the form $\Prj(1,1,1,1,...,1,a)$.  We do so in the following corollary.

\begin{cor}\label{khp111111a}

Consider the $d$-dimensional weighted projective space $\Prj(1,1,1,1,...,1,a)$, with $a \geq 2$.  Then $\KH_{n}(\Prj(1,1,1,1,...,1,a)) = 0$ for $n \leq -1$ and $\KH_{0}(\Prj(1,1,1,1,...,1,a)) = \Z^{d+1}.$

\end{cor}


\begin{proof}

We begin by giving a resolution of its singularities for $\Prj(1,1, ... ,1,a)$.  The fan for $\Prj(1,1,1,1,...,1,a)$ is generated by the $1$-dimensional cones $\{ e_{1}, e_{2}, ..., e_{d}, -e_{1}-e_{2}- \cdots -e_{d-1}-a e_{d} \}.$  The only singular cone is $\langle e_{1}, e_{2}, ..., e_{d-1}, -e_{1}-e_{2}- \cdots -e_{d-1}-a e_{d} \rangle$, and after refining our fan by adding the cone generated by $-e_{d}$, we get a smooth toric variety; call this smooth variety $\widetilde{X}$.  Now notice that the star of the cone $-e_{d}$ is just the fan for $\Prj^{d-1}$, so we get the blow-up square $$
\xymatrix{\Prj^{d-1} \ar[r]^-{i} \ar[d] & \widetilde{X} \ar[d]  \\
\{ * \} \ar[r] & \Prj(1,1,1,1,...,1,a)}
$$ and we obtain a resolution of singularities for $\Prj(1,1, ... ,1,a)$, as desired.

Let $i: \Prj^{d-1} \longrightarrow \widetilde{X}$ be as above, let $\pi: \widetilde{X} \longrightarrow \Prj^{d-1}$ denote the morphism induced by the lattice morphism $(x_{1}, ..., x_{d-1}, x_{d}) \mapsto (x_{1}, ..., x_{d-1})$, and let $f = \pi \circ i$.  To understand $f$, begin by picking an element $z \in \Z^{d-1}$.  Applying the map $i$, this corresponds to a ``line'' in $\Z^{d}$, given by $(z,t)$ for $t \in \Z$.  Then under $\widetilde{\pi}$, this line again maps to $z$. Applying the appropriate functors, we see that $f$ is an isomorphism.  Now since $\KH$ satisfies cdh descent, it gives rise to a long exact sequence $$
\xymatrixcolsep{0.6cm}\xymatrix{\cdots \ar[r] & \KH_{n}(\Prj(1,1,1,1,...,1,a)) \ar[r] & \KH_{n}(\widetilde{X}) \oplus \KH_{n}(k)  \ar[r]^-{\alpha_{n}} & \\
 \KH_{n}(\Prj^{d-1}) \ar[r] & \KH_{n-1}(\Prj(1,1,1,1,...,1,a)) \ar[r] & \cdots}
$$  Let $k$ denote the residue field associated to the point $\{ * \}$.  Our goal is to show that $\alpha_{n}$ is surjective.  Since $\alpha_{n}$ is the difference of the morphism $i_{n}^{*}: \KH_{n}(\widetilde{X}) \longrightarrow \KH_{n}(\Prj^{d-1})$ and the morphism $j_{n}^{*}: \KH_{n}(k) \longrightarrow \KH_{n}(\Prj^{d-1})$, it is enough to show that $i_{n}^{*}$ is surjective.  But $f = \pi \circ i$ was shown to be an isomorphism, so for every $n$, $i_{n}^{*}$ is indeed surjective.  By exactness, $\KH_{n}(\Prj(1,1,1,1,...,1,a))$ is a subgroup of $\KH_{n}(\widetilde{X}) \oplus \KH_{n}(k)$ for every $n$; therefore, $\KH_{n}(\Prj(1,1,1,1,...,1,a)) = 0$ for $n \leq -1$.  In the case $n=0$, since $\widetilde{X}$ is a smooth, projective toric variety, we have that $\KH_{0}(\widetilde{X})$ is a free abelian group of finite rank by \cite[Corollary $7.8$]{MP}; therefore, $\KH_{0}(\Prj(1,1,1,1,...,1,a))$ is itself free abelian.  By Theorem \ref{main2}, we have $\KH_{0}(\Prj(1,1,1,1,...,1,a)) = \Z^{d+1}$, as desired.


\end{proof}

\section{The $\mathcal{F}_{\K}$ groups for Complete Simplicial Toric Surfaces and for Weighted Projective Spaces $\Prj(a,b,c)$}\label{sect:fkpabc}

Having calculated the $\KH(-) \otimes \Q$ groups for weighted projective spaces in Theorem \ref{main2}, we are now ready to examine the $\mathcal{F}_{\K}$ groups.  Recall from \cite[Theorem $1.6$]{CHW} that $(\mathcal{F}_{K})_{n}(X) = H_{Zar}^{-n}(X, \mathcal{F}_{\HC}[1])$; see \cite[Definition $1.4$]{CHW} for the definition of $\mathcal{F}_{\HC}$.  The results of \cite{CHW} only hold if our varieties are defined over a field of characteristic $0$; given this, we will restrict ourselves to this case for the remainder of the paper.

In \cite{CHSW}, the authors prove that if $k$ is a field of characteristic $0$ and $X$ a $k$-scheme essentially of finite type and of dimension $d$, then $X$ is $\K_{-d}$-regular and for $n < -d$, we have $\K_{n}(X) = 0$ .  Our goal will be to derive stronger $\K$-regularity results for complete toric varieties.  We begin with the case of complete toric surfaces, and then extend these results to higher dimensional complete toric varieties (satisfying extra conditions) in Section \ref{sect:fkdimn}.  We begin by with the following theorem.




\begin{thm}[$\mathcal{F}_{\K}$ Decomposition Theorem]\label{fkdecomp}

Let $X$ be any complete toric surface, and let $U_{\sigma_{1}}$, $U_{\sigma_{2}}$,...,$U_{\sigma_{m}}$ be all the open sets associated to a maximal cone in the fan $\Delta_{X}$.  Then we have \bea (\mathcal{F}_{\K})_{n}(X) \cong (\mathcal{F}_{\K})_{n}(U_{\sigma_{1}}) \oplus (\mathcal{F}_{\K})_{n}(U_{\sigma_{2}}) \oplus \cdots \oplus (\mathcal{F}_{\K})_{n}(U_{\sigma_{m}}) \eea for all $n$.

\end{thm}

\begin{proof}

We proceed by induction on the number of open sets associated to maximal cones.  Let $X = U_{\sigma_{1}} \cup U_{\sigma_{2}}$.  Covering $X$ by $U_{\sigma_{1}}$ and $U_{\sigma_{2}}$ and using Zariski descent, we have the long exact sequence: \bea \cdots \longrightarrow & (\mathcal{F}_{\K})_{n}(X) & \longrightarrow (\mathcal{F}_{\K})_{n}(U_{\sigma_{1}}) \oplus (\mathcal{F}_{\K})_{n}(U_{\sigma_{2}})  \longrightarrow \cdots \nonumber \\
\cdots \longrightarrow & (\mathcal{F}_{\K})_{n}(U_{\sigma_{1}} \cap U_{\sigma_{2}}) & \longrightarrow (\mathcal{F}_{\K})_{n-1}(X) \longrightarrow \cdots \eea  Since $U_{\sigma_{1}} \cap U_{\sigma_{2}} = U_{\sigma_{1} \cap \sigma_{2}}$ is smooth (by normality), $(\mathcal{F}_{\K})_{n}(U_{\sigma_{1}} \cap U_{\sigma_{2}}) = 0$ for all $n$.  The result then follows by exactness.


Now suppose the result is true for all $k < m$.  Cover $X$ by $Y = \bigcup_{i=1}^{m-1} U_{\sigma_{i}}$ and $U_{\sigma_{m}}$, and let $Z = Y \cap U_{\sigma_{m}}$.  It is an easy exercise to see that $(\mathcal{F}_{\K})_{n}(Z) = 0$ for all $n$.  Using Zariski descent, we get that \bea (\mathcal{F}_{\K})_{n}(X) \cong (\mathcal{F}_{\K})_{n}(Y) \oplus (\mathcal{F}_{\K})_{n}(U_{\sigma_{m}}) \eea  for all $n$.  Now the result follows by applying our inductive hypothesis to $Y$.



\end{proof}







So our problem reduces to calculating $(\mathcal{F}_{\K})_{n}(U_{\sigma_{i}})$, for all $i$.  Doing so is in most cases extremely difficult; however, this does provide us with an approachable way to determine if $\Prj(a,b,c)$ is $\K_{0}$-regular.  To make this determination, we use the following result, due to Gubeladze.

\begin{lem}[Gubeladze]\label{Gulem}

For any regular ring $R$ and any monoid $M$, we have $\K_{n}(R) = \K_{n}(R[M]) = 0$ for $n \leq -1$.  Therefore, if $X = U_{\sigma}$ is an affine toric variety, then $\K_{0}(X) = \Z$ and $\K_{n}(X) = 0$ for $n \leq -1$.  Consequently, $(\mathcal{F}_{\K})_{n}(X) = 0$ for $n \leq 0$.

\end{lem}

\begin{proof}

See \cite[Theorem $1.3$]{Gub}.  The $\K_{0}$ part may also be found in \cite[Proposition $5.7$]{CHWW}.

\end{proof}

\begin{cor}\label{fkdim2cor1}

Any complete toric surface is $\K_{0}$-regular.

\end{cor}

\begin{proof}

This is immediate from Theorem \ref{fkdecomp} and Lemma \ref{Gulem}.

\end{proof}

\begin{rek}

Applying Corollary \ref{fkdim2cor1} to the weighted projective space $\Prj(a,b,c)$ establishes part (b) of Theorem \ref{bigmain1}.

\end{rek}



















\section{The $\mathcal{F}_{\K}$ groups for Weighted Projective Spaces of Higher Dimensions}\label{sect:fkdimn}

Unfortunately, the techniques of Section \ref{sect:fkpabc} do not extend to higher dimensions in general.  The problem that arises is that, while $U_{\sigma_{i}} \cap U_{\sigma_{j}}$ was smooth in dimension $2$ (by normality), $U_{\sigma_{i}} \cap U_{\sigma_{j}}$ need not be smooth in dimensions $d>2$.

\begin{exa}\label{singulardim3}

Consider the $3$-dimensional weighted projective space $\Prj(1,1,2,4)$.  The fan is given by all subsets of the set of $1$-dimensional cones $\{ (1,0,0), (0,1,0), (0,0,1), (-1,-2,-4) \}.$  Consider the two cones $\sigma_{1} = \langle (1,0,0), (0,1,0), (-1,-2,-4) \rangle$ and $\sigma_{2} = \langle (1,0,0), (0,0,1), (-1,-2,-4) \rangle.$  The intersection $\sigma_{1} \cap \sigma_{2} = \langle (1,0,0), (-1,-2,-4) \rangle$ is singular.  Indeed, for this cone to be smooth, we would need a vector $(a,b,c) \in \Z^{3}$ such that the matrix \be \left(\begin{array}{ccc}
1 &  -1 &  a \\
0 &  -2 &  b \\
0 &  -4 &  c  \\
\end{array}\right) \ee has determinant $\pm 1$.  But this is impossible since the determinant of this matrix is $4b-2c$.  

\end{exa}

As Example \ref{singulardim3} shows, $U_{\sigma_{i}} \cap U_{\sigma_{j}}$ need not be smooth even in dimension $3$, and so we cannot express $(\mathcal{F}_{\K})_{n}(X)$ as a direct sum of the $(\mathcal{F}_{\K})_{n}(U_{\sigma_{i}})$'s.  However, if we impose additional conditions on $X$, we can still recover an analog of Theorem \ref{fkdecomp} in dimensions $d > 2$.

\begin{thm}\label{fkdecomp2}

Let $X$ be a complete toric variety of dimension $d > 2$, and suppose that the dimension of the singular set of $X$ is $0$ (that is, $X$ is smooth in all codimensions $\leq d-1$).  Let $U_{\sigma_{1}}$, $U_{\sigma_{2}}$,...,$U_{\sigma_{m}}$ be all the open sets associated to a maximal cone in the fan $\Delta_{X}$.  Then we have \bea (\mathcal{F}_{\K})_{n}(X) \cong (\mathcal{F}_{\K})_{n}(U_{\sigma_{1}}) \oplus (\mathcal{F}_{\K})_{n}(U_{\sigma_{2}}) \oplus \cdots \oplus (\mathcal{F}_{\K})_{n}(U_{\sigma_{m}}) \eea for all n.

\end{thm}

\begin{proof}

First notice that the dimension of a singular cone is precisely the codimension of the singularities created by that cone.  So the statement that $X$ is smooth in all codimensions $\leq d-1$ is equivalent to saying that the only possible singular cones of $\Delta_{X}$ are maximal cones.  With this observation, the proof is almost word-for-word the same as the proof for Theorem \ref{fkdecomp}.  The only difference is that, for $Z = Y \cap U_{\sigma_{m}}$, the fact that $(\mathcal{F}_{\K})_{n}(Z)=0$ for all $n$ isn't as easy to see.  However, observe that $Z$ is covered by open sets $U_{\sigma_{i} \cap \sigma_{m}}$ which are smooth by assumption, and that the intersection of any two of these is also smooth.  Now proceed by induction on the size of the cover of $Z$; we leave the details as an exercise to the interested reader.

\end{proof}

Using Theorem \ref{fkdecomp2}, we can now derive results that are analogous to those proven at the end of Section \ref{sect:fkpabc}.

\begin{cor}\label{fkdimdcor1}

Let $X$ be a complete toric variety of dimension $d > 2$, and suppose that the dimension of the singular set of $X$ is $0$.  Then $X$ is $\K_{0}$-regular.

\end{cor}

\begin{proof}

This is immediate from Theorem \ref{fkdecomp2} and Lemma \ref{Gulem}.

\end{proof}

\begin{rek}

Applying Corollary \ref{fkdimdcor1} to any weighted projective space $\Prj(q_{1},q_{2}, ... ,q_{d})$ where the singular set consists of only isolated points establishes part (c) of Theorem \ref{bigmain1}.

\end{rek}

Corollary \ref{fkdimdcor1} gives us a way to examining non-trivial classes of higher dimensional weighted projective spaces, as the following example demonstrates.

\begin{exa}\label{exap1abcd}

Consider the $d$-dimensional weighted projective space $\Prj(1,q_{1},q_{2}, ... ,q_{d})$ where $\gcd(q_{i},q_{j}) = 1$ for $i \neq j$.  The fan is generated by the $1$-dimensional cones $\{ e_{1}, e_{2}, ..., e_{d}, -q_{1}e_{1}-q_{2}e_{2}- \cdots -q_{d}e_{d} \}.$  As always, every $1$-dimensional cone is smooth, and obviously every cone involving only the $e_{i}$'s are smooth also.  So the only possibly non-smooth cones are those involving the cone $-q_{1}e_{1}-q_{2}e_{2}- \cdots -q_{d}e_{d}$.  To show the singular set consists of only isolated singular points, we need to consider non-maximal cones involving $-q_{1}e_{1}-q_{2}e_{2}- \cdots -q_{d}e_{d}$ and see that they are still smooth.  Let us consider the cone $\sigma = \langle e_{i_{1}}, e_{i_{2}}, ..., e_{i_{k}}, -q_{1}e_{1}-q_{2}e_{2}- \cdots -q_{d}e_{d} \rangle.$  Notice that $k \leq d-2$ since if $k=d-1$ then $\sigma$ would be a maximal cone.  Also notice that if $\sigma$ is shown to be smooth whenever $k=d-2$, then it is smooth for all choices of $k$.  So we can assume $k=d-2$.  Without loss of generality, suppose that $e_{i_{j}} = e_{j}$, so that we have $\sigma = \langle e_{1}, e_{2}, ..., e_{d-2}, -q_{1}e_{1}-q_{2}e_{2}- \cdots -q_{d}e_{d} \rangle.$

For $\sigma$ to be smooth, we need to be able to find a vector $(\alpha_{1},\alpha_{2}, ..., \alpha_{d}) \in \Z^{d}$ such that the matrix \be\left(\begin{array}{ccccccc}
1    &  0   &  0   &  \cdots  &  0   &  -q_{1}  &  \alpha_{1}  \\
0    &  1   &  0   &  \cdots  &  0   &  -q_{2}  &  \alpha_{2}  \\
0    &  0   &  1   &  \cdots  &  0   &  -q_{3}  &  \alpha_{3}  \\
\vdots & \vdots & \vdots &  \ddots  & \vdots & \vdots &  \vdots  \\
0    & 0    & 0    &  \cdots  & 1    & -q_{d-2}   &  \alpha_{d-2}   \\
0    & 0    & 0    &  \cdots  & 0    & -q_{d-1}   &  \alpha_{d-1}  \\
0    & 0    & 0    &  \cdots  & 0    & -q_{d}   &  \alpha_{d}
  \end{array}\right)
\ee has determinant $\pm 1$.  The determinant is $\alpha_{d-1}q_{d}-\alpha_{d}q_{d-1}$; since $\gcd(q_{d-1},q_{d}) = 1$, we may choose $\alpha_{d-1}$ and $\alpha_{d}$ such that $\alpha_{d-1}q_{d}-\alpha_{d}q_{d-1} = 1$.  Taking those choices for $\alpha_{d-1}$ and $\alpha_{d}$ and letting $\alpha_{i}=0$ for $i \leq d-2$ gives us the desired extension.  The argument is analogous for all other possible choices for $\sigma$.  Therefore, provided that $\gcd(q_{i},q_{j}) = 1$ for $i \neq j$, $\Prj(1,q_{1},q_{2}, ... ,q_{d})$ satisfies the conditions for Corollary \ref{fkdimdcor1}, and therefore is $\K_{0}$-regular.



\end{exa}


\begin{thm}\label{kp111111a}

Consider the $d$-dimensional weighted projective space $\Prj(1,1,1, ... ,1,a)$, where $a \geq 2$.  Then $\K_{n}(\Prj(1,1,1, ... ,1,a)) = 0$ for $n \leq -1$ and $\K_{0}(\Prj(1,1,1, ... ,1,a)) = \Z^{d+1}.$

\end{thm}

\begin{proof}

Corollary \ref{fkdimdcor1} and our work in Example \ref{exap1abcd} shows that $\Prj(1,1,1, ... ,1,a)$ is $\K_{0}$-regular. Applying Corollary \ref{khp111111a} then gives us the result.

\end{proof}

\begin{rek}

Theorem \ref{kp111111a} establishes part (d) of Theorem \ref{bigmain1}, and therefore completes the proof of Theorem \ref{bigmain1}.

\end{rek}

\section{Additional $\K$-regularity Results}\label{sect:K-regfail}

We conclude this paper by making two final observations regarding $\K$-regularity of weighted projective spaces.  We begin by observing that weighted projective spaces are not $\K_{1}$-regular in general, via the following theorem.

\begin{thm}\label{notk1reg}

The weighted projective space $\Prj(1,1,2)$ is not $\K_{1}$-regular.

\end{thm}

\begin{proof}

By Theorem \ref{fkdecomp}, $(\mathcal{F}_{\K})_{n}(\Prj(1,1,2)) \cong (\mathcal{F}_{\K})_{n}(U_{\sigma})$, where $\sigma = \langle (1,0), (-1,-2) \rangle$.  But $U_{\sigma} = \Spec \left(k[u,v,w]/\langle uw-v^{2} \rangle\right)$, and by \cite[Theorem $4.3$]{CHWW2}, $(\mathcal{F}_{\K})_{1}\left(k[u,v,w]/\langle uw-v^{2} \rangle\right) \neq 0$, giving the result.

\end{proof}

However, if the dimension of our weighted projective space is bigger than $2$ and the singular set has dimension bigger than $0$, then our weighted projective space need not be $\K_{0}$-regular either.  To see this requires different techniques than what we've presented so far.  We say that $X$ is $\K_{N}$-regular if, following the notation of \cite{CHW}, $\mathcal{F}_{\K}(X)$ is $N$-connected; that is, if $(\mathcal{F}_{\K})_{n}(X) = 0$ for all $n \leq N$. $N$-connectivity for $\mathcal{F}_{\HH}$ and $\mathcal{F}_{\HC}$ is defined similarly.  Then by the definition of $\mathcal{F}_{\K}$, the SBI-sequence, and \cite[Lemma $1.5$]{CHW}, we conclude that $\mathcal{F}_{\K}(X)$ is $0$-connected if and only if $\mathcal{F}_{\HH}(X)$ is $-1$-connected.

\begin{thm}

If $X$ is the projective space $\Prj(1,1,2,4)$, then $X$ is $\K_{-1}$-regular.  Furthermore, $\mathcal{F}_{\HH}(X)$ is not $-1$-connected.  As a consequence, $X$ is not $\K_{0}$-regular, and $(\mathcal{F}_{\K})_{0}(X) \neq 0$.

\end{thm}

\begin{proof}

Begin by observing that $X$ has four affine open sets associated to maximal cones.  These sets are the prime spectra of the following rings: $k[x,y,y\inv,z,z\inv]$, $k[x,x\inv,y,z,z\inv]$, $k[x,x\inv,y,y\inv,z]$, and $k[x\inv,y\inv,z\inv,x^{-4}z,y^{-2}z]$.  The proof that $X$ is $\K_{-1}$-regular works almost the same as the proof of Theorem \ref{fkdecomp} and is left as an exercise.  To see that $\mathcal{F}_{\HH}(X)$ is not $-1$-connected, consider the long exact sequence \bea \cdots \longrightarrow (\mathcal{F}_{\HH})_{-1}(X) \longrightarrow \HH_{-1}(X) \longrightarrow \pi_{-1}\Hyp_{cdh}(X, \HH) \longrightarrow \cdots \eea Now recalling the Hodge Decomposition for Hochschild homology from \cite{Wei2}, we have that \bea \pi_{-1}\Hyp_{cdh}(X, \HH) = H_{cdh}^{1}(X, \mathcal{O}_{X}) \oplus H_{cdh}^{2}(X, \Omega_{X}^{1}) \oplus H_{cdh}^{3}(X, \Omega_{X}^{2}) \nonumber \\
\HH_{-1}(X) = H_{Zar}^{1}(X, \mathcal{O}_{X}) \oplus H_{Zar}^{2}(X, \Omega_{X}^{1}) \oplus H_{Zar}^{3}(X, \Omega_{X}^{2}) \oplus A.Q. \nonumber \eea where A.Q. denotes the Andr\'{e}-Quillen homology pieces.  $X$ is given by taking the quotient of $\Prj^{3}$ by the action of the group $G = \Z/2\Z \oplus \Z/4\Z$; therefore, $H_{cdh}^{i}(X, \Omega_{X}^{i-1}) = [H_{cdh}^{i}(\Prj^{3}, \Omega_{\Prj^{3}}^{i-1})]^{G} = 0$ for all $i$.  So we can conclude that $(\mathcal{F}_{\HH})_{-1}(X) \neq 0$ if we can show $\HH_{-1}(X) \neq 0$.  Consider the group $H_{Zar}^{3}(X, \Omega_{X}^{2})$.  Define the following modules: $M_{1} = k[x,y,y\inv,z,z\inv]\langle dx \wedge dy, dx \wedge dz, dy \wedge dz \rangle$ and similarly with $M_{2}$ and $M_{3}$ (with rings $k[x,x\inv,y,z,z\inv]$ and $k[x,x\inv,y,y\inv,z]$, respectively).  Let $M_{4}$ be the module over the ring $k[x\inv,y\inv,z\inv,x^{-4}z,y^{-2}z]$ generated by the wedge products of any two of the five differentials $\{ d(x\inv), d(y\inv), d(z\inv), d(x^{-4}z), d(y^{-2}z) \}$, and let $M$ be the module $k[x,x\inv,y,y\inv,z,z\inv]\langle dx \wedge dy, dx \wedge dz, dy \wedge dz \rangle$.  The relevant portion of the \v{C}ech complex is:
$$
\xymatrix{ \cdots \ar[r] & M_{1} \times M_{2} \times M_{3} \times M_{4} \ar[r]^-{d}  &  M \ar[r] & 0}
$$ Consider the element $\frac{1}{xyz} dx \wedge dy \in M$.  I claim this is not in the image of $d$.  Indeed, observe that $d(f_{1}, f_{2}, f_{3}, f_{4}) = f_{1}-f_{2}+f_{3}-f_{4}$.  If we let $g_{i}$ be the $dx \wedge dy$ term for $f_{i}$, where $i=1,2,3$, and we let \bea f_{4} = & g_{4} & d(x\inv) \wedge d(y\inv) + h_{4} d(x^{-4}z) \wedge d(y\inv)+ \nonumber \\
& p_{4} & d(x\inv) \wedge d(y^{-2}z) + q_{4} d(x^{-4}z) \wedge d(y^{-2}z) \nonumber \eea (any additional terms in $f_{4}$ will not map to a $dx \wedge dy$ term in $M$, so we may exclude them), then we have that the image of $d$ has $dx \wedge dy$ terms of the form \bea\label{imaged} \left( g_{1}-g_{2}+g_{3}-\frac{1}{x^{2}y^{2}}g_{4}-\frac{4z}{x^{5}y^{2}}h_{4} - \frac{2z}{x^{2}y^{3}}p_{4} - \frac{8z^{2}}{x^{5}y^{3}}q_{4} \right) dx \wedge dy. \eea  We would need Expression \ref{imaged} to be $\frac{1}{xyz}$ for some choices of the $g_{i}$'s, $h_{4}$, $p_{4}$, and $q_{4}$.  However, this is impossible; indeed, we simply observe that any terms involving $g_{4}$, $h_{4}$, $p_{4}$, and $q_{4}$ all have powers of $x$ and $y$ in the denominator that are too large, that terms involving $g_{1}$, $g_{2}$, and $g_{3}$ do not include $\frac{1}{xyz}$, and that Laurant polynomials are determined by their coefficients.  Consequently, $H_{Zar}^{3}(X, \Omega_{X}^{2}) \neq 0$, implying that $\HH_{-1}(X) \neq 0$, that $(\mathcal{F}_{\HH})_{-1}(X) \neq 0$, and that $X$ is not $\K_{0}$-regular.  Since $X$ is $\K_{-1}$-regular, this implies that $(\mathcal{F}_{\K})_{0}(X) \neq 0$ as claimed.

\end{proof}

\section*{Acknowledgements} I would like to thank Christian Haesemeyer, whose hours of dedication and infinite patience in guiding me has made this work possible, and Mark Walker, for his helpful comments on early drafts of this paper.  I would also like to thank both Justin Shih and Kenneth Maples for the many enlightening discussions and helpful comments that they both provided.  Finally, I would like to thank an anonymous referee for his suggested improvement to Theorem \ref{fixedpoint4} and for pointing out the connection between this paper and \cite{Gub2}.

\end{document}